\documentclass[11pt]{article}
\usepackage{amsmath,amsfonts,amsthm,amscd,amssymb,graphicx}
\usepackage{subfigure}

\numberwithin{equation}{section}

\usepackage{hyperref}
\usepackage{verbatim} 

\newtheorem{theorem}{Theorem}[section]

\newtheorem{lemma}[theorem]{Lemma}

\newtheorem{proposition}[theorem]{Proposition}
\newtheorem{prop}[theorem]{Proposition}

\newtheorem{remark}[theorem]{Remark}

\newcommand{\RR}{\mathbb{R}}

\newcommand{\CC}{\mathbb{C}}

\newcommand{\ZZ}{{\mathbb Z}}
\newcommand{\TT}{{\mathbb T}}

\def\beq{\begin{equation}}
\def\eeq{\end{equation}}
\def\bb1{{1\!\!1}}
\def\R{\mbox{Re }}

\def\triangle{\Delta}
\def\bega{\begin{aligned}}
\def\enda{\end{aligned}}

\begin{document}

\title{The inviscid limit of Navier-Stokes equations for analytic data on the half-space} 

\author{Toan T. Nguyen\footnotemark[1] \and Trinh T. Nguyen\footnotemark[1]
}

\maketitle

\renewcommand{\thefootnote}{\fnsymbol{footnote}}

\footnotetext[1]{Department of Mathematics, Penn State University, State College, PA 16803. Emails: nguyen@math.psu.edu; txn5114@psu.edu.}



\begin{abstract}

In their classical work \cite{SammartinoCaflisch1,SammartinoCaflisch2}, Caflisch and Sammartino proved the inviscid limit of the incompressible Navier-Stokes equations for well-prepared data with analytic regularity in the half-space. Their proof is based on the detailed construction of Prandtl's boundary layer asymptotic expansions. In this paper, we give a direct proof of the inviscid limit for general analytic data without having to construct Prandtl's boundary layer correctors. Our analysis makes use of the boundary vorticity formulation and the abstract Cauchy-Kovalevskaya theorem on analytic boundary layer function spaces that capture unbounded vorticity. 

\end{abstract}



\section{Introduction}
In this paper, we are interested in the inviscid limit of the Navier-Stokes equations for incompressible fluids
\beq \label{NS}
\begin{aligned} 
\partial_{t}u+u\cdot\nabla u+\nabla p&=\nu\triangle u
\\
\nabla\cdot u&=0
\end{aligned}
\eeq 
posed on the half space $\TT \times \RR_+$, with the classical no-slip boundary condition
\beq \label{NS-BC}
\begin{aligned}
u_{\vert_{z=0}} &=0 .
\end{aligned}
\eeq 
In the inviscid limit: $\nu \to 0$, one would expect that solutions $u^\nu$ converge to solutions of the corresponding Euler equations for incompressible fluids.

However, the inviscid limit problem for the no-slip boundary condition \eqref{NS-BC} is open due to the appearance of boundary layers and the creation of unbounded vorticity near the boundary. On the one hand, the friction causes the fluid to stick to the boundary, the no-slip condition \eqref{NS-BC}. On the other hand, the inviscid flow allows the fluid to slip along the boundary. The rapid change of the tangential velocity on the boundary gives rise to transition or boundary layers in the small viscosity limit. As a consequence, the vorticity is of order $\delta_\nu^{-1}$ on the boundary, in which $\delta_\nu$ denotes the thickness of boundary layers, within which the velocity changes rapidly. This leads to a possible large energy production, due to the large convection $u \cdot \nabla u$, within the boundary layers. 

In his seminal paper \cite{Pra:1904}, Prandtl postulated a boundary layer ansatz that balances the inertial and viscous forces in the dynamics of \eqref{NS}, leading to the well-known Prandtl's boundary layer equation, a simplification of complicated Navier-Stokes equations in a thin layer of thickness $\delta_\nu = \sqrt \nu$. In the half space, the Prandtl's ansatz reads
\beq \label{Ansatz}
u(t,x,z) =   u_E(t,x,z) + u_P \Bigl(t,x, {z \over \sqrt{\nu}}\Bigr) + o(1)_{L^\infty}
\eeq
where $u_E$ solves the corresponding Euler equations, and $u_P$ is introduced to correct the no-slip boundary condition of the Navier-Stokes equations which does not satisfy by Euler solutions $u_E$. Here, the velocity field  $u_E(t,x,0) + u_P(t,x, Z)$ solves the well-known Prandtl's boundary layer equation, and the remainder $o(1)_{L^\infty}$ is expected to converge to zero in $L^\infty$ as $\nu \to 0$. Formally, it is even possible to write a higher order asymptotic
expansion for $u^\nu$ in terms of powers of $\sqrt{\nu}$. Since then, the Prandtl boundary layers have been intensively studied in the mathematical literature. Notably, solutions to the Prandtl equations 
have been constructed for monotonic data (\cite{Oleinik, Oleinikbook, Alex, MW}) or data with Gevrey or analytic regularity (\cite{SammartinoCaflisch1, GVMasmoudi15, VV}, among others). 
In the case of non-monotonic data with Sobolev regularity, the Prandtl boundary layer equations are ill-posed (\cite{GVDormy,GVN1,GN1}). 

The validity of Prandtl's Ansatz \eqref{Ansatz}, and hence the inviscid limit, were established in \cite{SammartinoCaflisch1,SammartinoCaflisch2}
for well-prepared data with analytic regularity (precisely, analytic data that initially satisfy the boundary layer expansion \eqref{Ansatz} with remainder of order $\sqrt\nu$). A similar result is also obtained in \cite{Maekawa2,Wang}. 
The Ansatz \eqref{Ansatz}, with a specific boundary layer shear profile, has been recently justified for data with Gevrey regularity \cite{GVMM}. When  only data with Sobolev regularity are assumed, except for data with special symmetry data with special symmetry (\cite{Maz1,Maz5} and the references therein), the asymptotic expansions \eqref{Ansatz} are false due to the strong instability of boundary layers (\cite{Grenier00CPAM,GrN4,GrN5,GrN2,GrN3}), including those boundary layers that are spectrally stable to the Euler equations. Finally, we mention that the inviscid limit holds for other types of boundary conditions other than the no-slip one \eqref{NS-BC}; see, for instance, \cite{IftimieSueur,MasRou1,Rugo}, and also the review papers \cite{BT,MaeM}. In particular, \cite{MasRou1} proves the inviscid limit for Navier-type boundary conditions via energy estimates without dealing directly with the boundary layer asymptotic expansions.   

On the other hand, back to the no-slip case \eqref{NS-BC}, Kato \cite{Kato} constructed a much thinner boundary layer of thickness $\delta_\nu = \nu$, as compared to the Prandtl's thickness $\delta_\nu = \sqrt \nu$, aiming precisely to control the large convection near the boundary. This leads to his well-known criterium, which asserts that the inviscid limit holds in the energy norm if and only if the energy dissipation near the boundary vanishes in the inviscid limit: precisely
\begin{equation}\label{Kato} \nu \int_0^T \iint_{\{z\lesssim \nu\}} |\nabla u|^2 \; dxdzdt \quad\longrightarrow \quad 0, \qquad \mbox{as}\quad \nu \to 0.\end{equation}
See also its many variants: for instance, \cite{BT,ConstVV, CKV,ConVV17,Kel}, among others.

In this paper, we give a direct proof of the inviscid limit for data with analytic regularity without having to construct Prandtl's boundary layer correctors. The proof relies on the boundary vorticity formulation, the pointwise bounds on the Green function, and the abstract Cauchy-Kovalevskaya theorem on boundary layer function spaces.

\subsection{Boundary vorticity formulation}

We shall work with the boundary vorticity formulation and the solution representation as in the recent work by Maekawa \cite{Maekawa1,Maekawa2}; see also \cite{Anderson}. Precisely, let $\omega = \partial_z u_1 - \partial_x u_2$ be the corresponding vorticity in $\TT\times \RR_+$. Then, the vorticity equation reads
\beq\label{NS-vor}
\partial_{t}\omega-\nu\Delta\omega=-u\cdot\nabla\omega
\eeq
with $u = \nabla^\perp \Delta^{-1} \omega$. Here and throughout the paper, $\Delta^{-1}$ denotes the inverse of the Laplacian operator with the Dirichlet boundary condition: precisely, $\phi = \Delta^{-1}\omega$ solves $\Delta \phi = \omega$ on the half-space $\TT \times \RR_+$, with $\phi_{\vert_{z=0}} =0$.

To ensure the no-slip boundary condition, we impose $\partial_t u_1 = 0$ on the boundary. This leads to 
$$ 0 = \partial_t u_1 = \partial_z \Delta^{-1} \partial_t \omega = \partial_z \Delta^{-1} (\nu \Delta \omega - u \cdot \nabla \omega)$$
on the boundary. 
Introduce $\omega_*$ so that $\Delta \omega_* = 0$ with $\omega_* = \omega$ on the boundary. This yields $\partial_z \Delta^{-1}\Delta \omega = \partial_z (\omega - \omega_*) = (\partial_z + |\partial_x|)\omega$, in which $|\partial_x|$ denotes the Dirichlet-to-Neumann operator on the half space. Thus, the boundary condition on vorticity reads
\beq\label{BC-vor}
\nu (\partial_z + |\partial_x|)\omega_{\vert_{z=0}} = [\partial_z \Delta^{-1} ( u \cdot \nabla \omega)]
_{\vert_{z=0}}.\eeq

Throughout this paper, we shall deal with the Navier-Stokes solutions that solve \eqref{NS-vor}-\eqref{BC-vor}, together with the Biot-Savart law $u = \nabla^\perp \Delta^{-1} \omega$. Such a solution will  be constructed via the Duhamel's integral representation, treating the nonlinearity as a source term.  

\subsection{Analytic boundary layer function spaces}\label{sec-defBL}

In this paper, we shall deal with analytic boundary layer spaces introduced in \cite{GrN2, GrNbook}. Precisely, we consider holomorphic functions on the pencil-like complex domain:
\begin{equation}\label{def-pencil}\Omega_{\sigma} =\Big\{z\in\mathbb{C}:\quad|\Im z|<\min\{\sigma \Re z ,\sigma \}\Big\} ,\end{equation}
for $\sigma>0$. Let $\delta = \sqrt \nu$ be the classical boundary layer thickness. We introduce the analytic boundary layer function spaces ${\cal B}^{\sigma,\delta}$ that consists of holomorphic functions on $\Omega_\sigma$ with a finite norm 
\begin{equation}\label{def-bl0}
\| f \|_{\sigma,\delta}  = \sup_{z\in \Omega_\sigma} | f(z) | e^{\beta \Re z} 
\Bigl( 1 + \delta^{-1} \phi_{P} (\delta^{-1} z)  \Bigr)^{-1}
\end{equation}
for some small $\beta>0$, and for boundary layer weight function $$ \phi_P(z) = \frac{1}{1+|\Re z|^P} $$
for some fixed constant $P>1$. Here, we suppress the dependence on $\beta,P$ as they are fixed throughout the paper. We expect that the vorticity function $\omega (t,x,z)$, for each fixed $t,x$, will be in ${\cal B}^{\sigma,\delta}$, precisely describing the behavior near the boundary and near infinity. In fact, there is an additional initial layer of thickness $\delta_t = \sqrt{\nu t}$ that appears near the boundary. To capture this, we introduce the time-dependent boundary layer norm: 
\begin{equation}\label{def-blt}
\| f\|_{\sigma,\delta(t)}  = \sup_{z\in \Omega_\sigma} | \omega(z) | e^{\beta \Re z} 
\Bigl( 1 + \delta_t^{-1} \phi_{P} (\delta_t^{-1} z)  +  \delta^{-1} \phi_{P} (\delta^{-1} z)  \Bigr)^{-1} ,
\end{equation}
with $\delta_t = \sqrt{\nu t}$, $\delta = \sqrt \nu$, and with the same boundary layer weight function $\phi_P(\cdot)$. By convention, the norm $\| \cdot \|_{\sigma,\delta(0)}$ at time $t=0$ is replaced by $\| \cdot\|_{\sigma,\delta}$, the boundary layer norm with precisely one boundary layer behavior of thickness $\delta$, and $\| \cdot \|_{\sigma,0}$ denotes the norm without the boundary layer behavior.

For functions depending on two variables $f(x,z)$, we introduce the partial Fourier transform
in variable $x$ 
$$f(x,z)=\sum_{\alpha\in\mathbb{Z}}f_{\alpha}(z)e^{i\alpha x}$$
and introduce the following analytic norm 
$$||f|| _{\rho,\sigma,\delta(t)}=\sum_{\alpha\in\mathbb{Z}} e^{\rho|\alpha|} ||f_{\alpha}||_{\sigma,\delta(t)}$$
for $\rho,\sigma>0$. We denote by $B^{\rho,\sigma,\delta(t)}$ the corresponding spaces. 
In Section \ref{sec-analyticBL}, we shall recall some basic properties of such analytic function spaces.

\subsection{Main results}

Our main results are as follows. 

\begin{theorem}\label{theo-main} Let $M_0>0$ and let $\omega_0$ be in ${\cal B}^{\rho_0,\sigma_0,\delta}$ for $\rho,\sigma>0$ and for $\delta = \sqrt \nu$, with $\|\omega_0 \|_{\rho_0,\sigma_0,\delta}\le M_0$. Then, there is a positive time $T$ so that the solution $\omega(t)$ to the Navier-Stokes equations \eqref{NS-vor}-\eqref{BC-vor}, with the initial data $\omega(0) = \omega_0$, exists in $C^1([0,T]; {\cal B}^{\rho,\sigma, \delta(t)})$ for $0<\rho <\rho_0$ and $0<\sigma<\sigma_0$. In particular, there is a $C_0$ so that the vorticity $\omega(t)$ satisfies 
\begin{equation}\label{bdry-propagate}|\omega(t,x,z)| \le C_0 e^{-\beta z} \Bigl( 1 + \delta_t^{-1} \phi_{P} (\delta_t^{-1} z)  +  \delta^{-1} \phi_{P} (\delta^{-1} z)  \Bigr)\end{equation}
for $(t,x,z)\in [0,T]\times \TT \times \RR_+$, with $\delta_t = \sqrt{\nu t}$ and $\delta = \sqrt \nu$. 
\end{theorem}

\begin{theorem}\label{theo-limit}  Let $M_0>0$ and let $u_0^\nu$ be divergence-free analytic initial data so that $u_0^\nu = 0$ on the boundary and $\omega_0^\nu = \nabla \times u_0^\nu$ is in ${\cal B}^{\rho_0,\sigma_0,\delta}$ for $\rho,\sigma>0$ and for $\delta = \sqrt \nu$, with $\|\omega_0^\nu\|_{\rho_0,\sigma_0,\delta}\le M_0$. Then, the inviscid limit holds for Navier-Stokes solutions with the initial data $u_0$. Precisely, there are unique local solutions $u^\nu(t)$ to the Navier-Stokes equations \eqref{NS}-\eqref{NS-BC}, for small $\nu>0$, and a unique solution $u^0(t)$ to the corresponding Euler equations, with initial data $u^0(0) = \lim_{\nu\to 0}u_0^\nu$, so that 
$$\sup_{t\in [0,T]}\| u^\nu(t) - u^0(t)\|_{L^p} \to 0$$
for $2\le p<\infty$, as $\nu \to 0$.  
\end{theorem}

As mentioned, the proof of the main theorems is direct, using the vorticity formulation \eqref{NS-vor}-\eqref{BC-vor}. In fact, the existence of analytic solutions is proved, without having to derive the pointwise bounds on the Green function for the Stokes problem; see Sections \ref{sec-convS} and \ref{sec-nonlinear}. However, in order to prove the propagation of boundary layer behaviors as described in \eqref{bdry-propagate}, the detailed estimates on the Green function are crucial. The main results apply in particular for well-prepared analytic data that satisfy the Prandtl's ansatz \eqref{Ansatz}.  For general analytic data, beside the Prandtl's layers, the initial layers whose thickness is of order $\sqrt{\nu t}$ appear as captured in \eqref{bdry-propagate}.

Finally, we mention that the analysis avoids dealing directly with the Prandtl's layers and Prandtl's asymptotic expansions, and hence appears robust to resolve the inviscid limit problem (for analytic data) in domains with curved boundaries.     

The paper proceeds with some basic properties of the analytic boundary layer norms and elliptic estimates in Section \ref{sec-analyticBL}. The main analysis of the paper is presented in Section \ref{sec-Stokes123} where we study in details the Stokes problem with boundary data and sources in the boundary layer function spaces.  The nonlinear iteration and the proof of the main theorems are given in Section \ref{sec-proof}. 


\section{Analytic function spaces}\label{sec-analyticBL}
In this section, we shall prove some basic properties of the analytic norms as well as the elliptic estimates that yield bounds on velocity in term of vorticity. These norms and estimates can be found in \cite{GrNbook}. See also \cite{SammartinoCaflisch1,SammartinoCaflisch2}. 

\subsection{Analytic spaces}\label{sec-analytic}
Let $f(x,z)$ be holomorphic functions on $\TT \times \Omega_{\sigma}$, with $\Omega_\sigma$ being the pencil-like complex domain defined as in \eqref{def-pencil}. 
For $\rho,\sigma>0$ and $p\ge 1$, we introduce the analytic function spaces denoted by $\mathcal L^p_{\rho,\sigma}$ with the finite norm 
\begin{equation}\label{def-Anorm} \| f\| _{\mathcal L^p_{\rho,\sigma}} :=\sum_{\alpha\in\mathbb{Z}} e^{\rho|\alpha|} \|f_{\alpha}\|_{L^p_\sigma}, \qquad \| f_\alpha \|_{L^p_\sigma} := \sup_{0\le \theta < \sigma}\Big( \int_{\partial\Omega_\theta} |f_\alpha(z)|^p\; |dz|\Big)^{1/p},\end{equation}
in which $f_\alpha = f_\alpha(z)$ denotes the Fourier transform of $f(x,z)$. In the case when $p=\infty$, we replace the 
$L^p$ norm by the sup norm over $\Omega_\sigma$. Recalling the analytic boundary layer space $B^{\rho,\sigma,\delta(t)}$ introduced in Section \ref{sec-defBL}, we have

\begin{lemma}\label{lem-emL1} There holds the embedding ${\cal B}^{\rho,\sigma,\delta(t)} \subset \mathcal L^1_{\rho,\sigma}$. 
\end{lemma}
\begin{proof}
For the holomorphic functions $f_\alpha(z)$ satisfying 
$$ |f_\alpha(z)| \le \| f_\alpha\|_{\sigma, \delta(t)}e^{-\beta \Re z} \Bigl( 1 + \delta_t^{-1} \phi_{P} (\delta_t^{-1} z)  +  \delta^{-1} \phi_{P} (\delta^{-1} z)  \Bigr),$$
it is clear that $\| f_\alpha \|_{L^1_\sigma} \le \| f_\alpha \|_{\sigma,\delta(t)}$. By taking the summation over $\alpha \in \ZZ$, the lemma follows. 
\end{proof}

\begin{lemma} \label{loss-1}
For any $0<\sigma'<\sigma$, $0<\rho'<\rho$, and $\psi(z)=\frac{z}{1+z}$, there hold
\begin{equation}\label{alg1}\| fg\|  _{\mathcal L^1_{\rho,\sigma}}  \le\| f\| _{\mathcal L^\infty_{\rho,\sigma}}\| g\|_{\mathcal L^1_{\rho,\sigma}}, \end{equation}
\begin{equation}\label{alg2}\|\partial_{x}f\|_{\mathcal L^1_{\rho',\sigma}}\le \frac{C}{\rho- \rho'}\| f\|_{\mathcal L^1_{\rho,\sigma}}, \qquad \| \psi(z)\partial_{z}f\| _{\mathcal L^1_{\rho,\sigma'}}\le \frac{C}{\sigma-\sigma'}\| f\|_{\mathcal L^1_{\rho,\sigma}}. \end{equation}
The same estimates hold for boundary layer norms $\| \cdot \|_{\rho,\sigma,\delta}$ replacing $\|\cdot \|_{\mathcal L^1_{\rho,\sigma}}$ in the above three inequalities. 
\end{lemma}

\begin{proof} 
By definition, we write $$fg(x,z)=\sum_{\alpha\in\mathbb{Z}}e^{i\alpha x}\sum_{\beta\in\mathbb{Z}}f_{\alpha-\beta}(z)g_{\beta}(z)$$
and hence, we estimate 
$$
\begin{aligned}
\| fg\|_{\mathcal L^1_{\rho,\sigma}}&=\sum_{\alpha\in\mathbb{Z}}e^{\rho|\alpha|}\|  \sum_{\beta\in\mathbb{Z}}f_{\alpha-\beta}(\cdot)g_{\beta}(\cdot)\| _{L^1_\sigma}
\le\sum_{\alpha\in\mathbb{Z}}\sum_{\beta\in\mathbb{Z}}e^{\rho|\alpha-\beta|}e^{\rho|\beta|}\| f_{\alpha-\beta}\| _{L^\infty_\sigma}\| g_{\beta}\|  _{L^1_\sigma}\\
\end{aligned}
$$
which proves the first inequality. As for the second, we compute 
$$\| \partial_{x}f\|_{\mathcal L^1_{\rho',\sigma}}\le\sum_{\alpha}\|  f_{\alpha}\|_{L^1_{\sigma}}|\alpha|e^{\rho'|\alpha|}$$
Using the fact that $(\rho-\rho')|\alpha|e^{(\rho'-\rho)|\alpha|}$ is bounded, the second inequality follows.

Finally, we check the third inequality. By the Cauchy integral formula, we have 
$$\partial_{z}f_\alpha(z)=\frac{1}{2\pi i}\int_{C(z,R_z)}\frac{f_\alpha(y)}{(y-z)^{2}}dy$$
where $C(z,R_z)$ is the circle, centered at $z$ and of radius $R_z$ so that $C(z,R_z) \in \Omega_{\sigma}$. 
Let us take 
\[
R_z=c_0(\sigma-\sigma')
\left\{ \begin{aligned}
\Re(z)\qquad &\text{if}\quad \Re(z)<1\\
1\qquad &\text{if}\quad \Re(z)\ge 1\\
\end{aligned}\right.
\]
for some small and positive $c_0$.
Thus, using the parametrization $y=z+e^{i w}R_z$ with $0\le w\le 2\pi$, we get 
\[
\bega
\partial_z f_\alpha(z)&=\frac{1}{2\pi i}\int_0^{2\pi}\frac{f_\alpha(z+R_ze^{iw})}{R_z^2 e^{2iw}}(R_z ie^{iw})dw\\
&=\frac{1}{2\pi R_z}\int_0^{2\pi}f_\alpha(z+R_ze^{iw})e^{-iw}dw
.\enda
\]
Now for any $0\le \theta ' < \sigma'$, we compute 
\[ \bega
\int_{\partial \Omega_{\theta'}}|\psi(\Re z)\partial_zf_\alpha(z)||dz|&\le \int_{\partial\Omega_{\theta'}}\int_0^{2\pi}\frac{\psi(\Re z)}{2\pi R_z}|f_\alpha(z+R_ze^{iw})||dw||dz|\\
&\le \frac{C_0}{\sigma-\sigma'} \int_{\partial\Omega_{\theta'}}\int_0^{2\pi}|f_\alpha(z+R_ze^{iw})||dw||dz|\\
&\le \frac{2\pi C_0}{\sigma-\sigma'} \sup_{0\le w\le 2\pi}\int_{\partial\Omega_{\theta'}}|f_\alpha(z+R_ze^{iw})||dz|.
\enda
\]

It remains to show that the above integral is bounded by $2\| f\|_{\mathcal L^1_{\rho,\sigma}}$. 
To this end, it suffices to show that for each fixed $w\in [0,2\pi]$, there is a positive constant $\theta <\sigma$ so that  
\begin{equation}\label{newline} z+R_ze^{iw}\in \partial\Omega_{\theta} , \qquad \forall ~z\in \partial\Omega_{\theta'}.\end{equation}
{\em Case 1: $\Re(z)\le 1$.} Recalling $R_z=c_0(\sigma-\sigma')\Re(z)$ and $\Im (z) = \theta' \Re(z)$ on $\partial \Omega_{\theta'}$, we compute 
\[ \bega
\Re(z+R_ze^{iw})&=\Re(z)+R_z\cos(w) = \Re(z)(1+c_0(\sigma-\sigma')\cos w)\\
\Im(z+R_ze^{iw})&=\Im(z)+R_z\sin(w) = \Re(z)(\theta'+c_0(\sigma-\sigma')\sin w).
\enda
\]
Hence, $z+R_ze^{iw}\in \partial\Omega_{\theta}$ for $\theta=\frac{\theta'+c_0(\sigma-\sigma')\sin w}{1+c_0(\sigma-\sigma')\cos w}. $ We now check $\theta<\sigma$. Indeed, we need 
\[
\theta'+c_0(\sigma-\sigma')\sin w< \sigma (1+c_0(\sigma-\sigma')\cos w)
\]
which is equivalent to \[
\sigma-\theta'> c_0(\sigma-\sigma')(\sin w-\sigma \cos w) .
\]
Since $\theta'<\sigma'$ and $c_0$ can be taken arbitrarily small (independent of $\sigma, \sigma'$), the above inequality follows. 

~\\
{\em Case 2: $\Re(z)\ge 1$.} Similarly, using $R_z=c_0(\sigma-\sigma')$ and $\Im(z)=\theta'$ on $\partial\Omega_{\theta'}$, we compute 

\[ 
\Im(z+R_ze^{iw})=\Im(z)+R_z\sin w=\theta'+c_0(\sigma-\sigma')\sin w=\widetilde\theta, 
\]
where $\widetilde \theta < \sigma $ for sufficiently small $c_0$ and for $\theta' < \sigma' < \sigma$. This proves \eqref{newline}. 
\end{proof}

\subsection{Elliptic estimates}
Next, we recall the elliptic estimates, which are adapted from \cite{GrNbook,GrN2}. 

\begin{prop} \label{inverseLaplace}
Let $\phi$ be the solution of
$
- \Delta \phi = \omega
$
with the zero Dirichlet boundary condition, and set $u = \nabla^\perp \phi$. Then, there hold 
\beq \label{laplace-5}
\| u_1\|_{\mathcal L^\infty_{\rho,\sigma}}+\| u_2\| _{\mathcal L^\infty_{\rho,\sigma}}\le C\| \omega\| _{\mathcal L^1_{\rho,\sigma}} ,
\eeq 
\beq 
\label{laplace-6}
\| \partial_{x}u_1\| _{\mathcal L^\infty_{\rho,\sigma}}+\| \nabla u_2\| _{\mathcal L^\infty_{\rho,\sigma}} + \| \psi^{-1}u_2\| _{L^\infty_{\rho,\sigma}} \le C\| \omega\|  _{\mathcal L^1_{\rho,\sigma}}+C\|  \partial_{x}\omega\| _{\mathcal L^1_{\rho,\sigma}} ,
\eeq 
\beq \label{laplace-7}
\| \nabla u_1 \|_{\mathcal L^1_{\rho,\sigma}} + \|\nabla u_2 \|_{\mathcal L^1_{\rho,\sigma}} \le C\| \omega\|  _{\mathcal L^1_{\rho,\sigma}}, 
\eeq 
 with $\psi(z) = z/ (1+z)$, for some constant $C$.
\end{prop}
\begin{proof} Taking the Fourier transform, it suffices to study the classical one-dimensional Laplace equation
\beq \label{Lap1}
\partial_z^2 \phi_\alpha - \alpha^2 \phi_\alpha = \omega_\alpha
\eeq
on $\Omega_{\sigma}$, with the Dirichlet boundary condition $ \phi_\alpha(0) = 0$, and $\alpha>0$. For real values $z$, the solution $\phi_\alpha $ of (\ref{Lap1}) is explicitly
given by 
$$\phi_\alpha(z)=\int_{0}^z G_-(y,z)\omega_\alpha(y)dy + \int_z^\infty G_+(y,z)\omega_\alpha(y)dy$$
with $$G_\pm(y,z) = -\frac{1}{2\alpha} \Big( e^{\pm\alpha (z-y) } - e^{-\alpha (y+z)} \Big).
$$
This expression may be extended to complex values of $z$. Indeed, for $z\in \Omega_\sigma$, there is a positive $\theta$ so that $z\in \partial\Omega_\theta$. We then write  $\partial\Omega_\theta = \gamma_-(z) \cup \gamma_+(z) $, consisting of complex numbers $y \in \partial\Omega_\theta$ so that $\Re y < \Re z$ and $\Re y > \Re z$, respectively. Then, we write 
\beq \label{laplace-3}
\begin{aligned}\phi_\alpha(z)&= \int_{\gamma_-(z)} G_-(y,z)\omega_\alpha(y)dy + \int_{\gamma_+(z)} G_+(y,z)\omega_\alpha(y)dy
.\end{aligned}\eeq
We note in particular that for $y \in \gamma_\pm(z)$, there holds 
$$ |G_\pm(y,z)|\le \alpha^{-1} e^{-\alpha |y-z|} .$$
This proves that 
\begin{equation}\label{bd-phiz} |\phi_\alpha(z)| \le \int_{\partial \Omega_\theta} \alpha^{-1} e^{-\alpha |y-z|} |\omega_\alpha(y)|\; |dy| \le  \alpha^{-1} \int_{\partial \Omega_\theta}  |\omega_\alpha(y)|\; |dy| ,\end{equation}
which by definition yields $\sup_{\Omega_\sigma} |\alpha \phi_\alpha(z)| \le \| \omega_\alpha\|_{L^1_\sigma}$. The same proof holds for $\partial_z \phi_\alpha(z)$. This completes the proof of \eqref{laplace-5}. The estimate \eqref{laplace-6} follows by treating $\partial_x$ as multiplication by $\alpha$ in the Fourier space. 

Finally, taking $L^1$ norm of the estimate \eqref{bd-phiz} and upon noting that the kernel $\alpha e^{-\alpha|y-z|}$ is bounded in $L^1$ norm, we obtain the estimate \eqref{laplace-7} for $\alpha^2 \phi_\alpha$. The second derivative in $z$, we use $\partial_z^2 \phi_\alpha = \alpha^2 \phi_\alpha + \omega_\alpha$. This completes the proof of the lemma. 
\end{proof}

\subsection{Bilinear estimates}

\begin{lemma}\label{lem-bilinear}
For any $\omega$ and $\tilde \omega$,
denoting by $v$ the velocity related to $\omega$, we have
$$
\begin{aligned}
 \| v\cdot \nabla \tilde \omega \|_{\mathcal L^1_{\rho,\sigma}} &\le C \| \omega\|_{\mathcal L^1_{\rho,\sigma}}\| \tilde \omega_x\|_{\mathcal L^1_{\rho,\sigma}} + C (\| \omega\|_{\mathcal L^1_{\rho,\sigma}} +\|\omega_x\|_{\mathcal L^1_{\rho,\sigma}} ) \| \psi(z)\partial_z\tilde \omega\|_{\mathcal L^1_{\rho,\sigma}}
 \\
  \| v\cdot \nabla \tilde \omega \|_{\rho,\sigma,\delta} &\le C \| \omega\|_{\rho,\sigma,\delta}\| \tilde \omega_x\|_{\rho,\sigma,\delta} + C (\| \omega\|_{\rho,\sigma,\delta} +\|\omega_x\|_{\rho,\sigma,\delta} ) \| \psi(z)\partial_z\tilde \omega\|_{\rho,\sigma,\delta}
 \end{aligned}$$
\end{lemma}
\begin{proof}
We write $$
(v \cdot \nabla) \tilde \omega = v_1 \partial_x \tilde \omega+ v_2 \partial_z \tilde \omega.
$$
Using \eqref{alg1} and \eqref{laplace-5}, the first term is clear. On the other hand, the second term is due to \eqref{alg1} and \eqref{laplace-6}. Finally, for the boundary layer norms, we note that 
$$\| fg\|  _{\rho,\sigma,\delta}  \le\| f\| _{\mathcal L^\infty_{\rho,\sigma}}\| g\|_{\rho,\sigma,\delta}.$$
The lemma thus follows, upon using Proposition \ref{inverseLaplace} and noting $\| f \|_{\mathcal L^1_{\rho,\sigma}} \le \| f \|_{\rho,\sigma,\delta}$.
\end{proof}

\section{The Stokes problem}\label{sec-Stokes123}
In this section, we study the inhomogenous Stokes problem 
\beq 
\begin{aligned}\label{stokes}
\omega_{t}-\nu\triangle\omega&=f(t,x,z), \qquad \mbox{in}\quad \TT \times \Omega_\sigma,\\
\nu (\partial_{z}+|\partial_{x}|)\omega &=g(t,x), \quad\qquad \mbox{on} \quad z=0,\\
\end{aligned}
\eeq 
together with the initial data $\omega_{\vert_{t=0}} = \omega_0$. Let $e^{\nu t B}$ denote the semigroup of the corresponding Stokes problem: namely, the heat equation $\partial_{t}\omega-\nu\Delta\omega = 0$ on $\TT \times \Omega_\sigma$ with the homogenous boundary condition $
\nu (\partial_z + |\partial_x|)\omega_{\vert_{z=0}} =0.
$ Solutions to the linear Stokes problem is then constructed via the following Duhamel's integral representation, which will be proved in the next subsection,  
 \beq\label{Duh-w}
\omega(t)=e^{\nu tB}\omega_{0} +\int_{0}^{t}e^{\nu(t-s)B} f(s) \; ds + \int_0^t \Gamma(\nu (t-s))g(s) \;ds
\eeq 
in which $\Gamma(\nu t) = e^{\nu tB}(g\mathcal{H}^1_{\mathbb{T}\times \{y=0\}})$, where $\mathcal{H}^1_{\mathbb{T}\times \{y=0\}}$ is the one-dimensional Hausdorff measure restricted on the boundary; precisely, see \eqref{semigroup-exp}-\eqref{semigroup-exp1} for the explicit construction of $e^{\nu tB}$ and $\Gamma(\nu t)$ in term of the Green function for the Stokes problem. 

In this section, we shall derive uniform bounds for the Stokes semigroup in analytic spaces, with the analytic norm 
$$\| \omega\|  _{\rho,\sigma,\delta(t)}=\sum_{\alpha\in\mathbb{Z}} e^{\rho|\alpha|} \|\omega_{\alpha}\|_{\sigma,\delta(t)}$$
with the boundary layer norm defined by 
\begin{equation}\label{def-blnormFull}
\| \omega_\alpha\|_{\sigma,\delta(t)}  = \sup_{z\in \Omega_{\sigma}} | \omega_\alpha(z) | e^{\beta \Re z} 
\Bigl( 1 + \delta_t^{-1} \phi_{P} (\delta_t^{-1} z)  +  \delta^{-1} \phi_{P} (\delta^{-1} z)  \Bigr)^{-1} ,
\end{equation}
in which the boundary thicknesses are $\delta_t = \sqrt{\nu t}$ and $\delta = \sqrt \nu $.  As for the initial data, the norm is measured by $\| \omega_\alpha\|_{\sigma,\delta(0)}$, which consists of precisely one boundary layer behavior whose thickness is $\delta = \sqrt \nu$. We introduce 
$$ ||| \omega(t) |||_{\rho,\sigma,\delta(t),k}  = \sum_{j+\ell \le k} \|\partial_x^j (\psi(z)\partial_z)^\ell \omega(t)\|_{\rho,\sigma,\delta(t)}$$
and 
$$ ||| \omega |||_{\mathcal{W}^{k,1}_{\rho,\sigma}} = \sum_{j+\ell \le k} \|\partial_x^j (\psi(z)\partial_z)^\ell \omega(t)\|_{\mathcal L^1_{\rho,\sigma}} .$$ 
We also denote $||| g|||_{\rho,k}$ the corresponding analytic norm for $g = g(x)$. We obtain the following key proposition. 
 
\begin{proposition}\label{prop-Stokes} Let $e^{\nu t B}$ be the semigroup for the linear Stokes problem, and $\Gamma(\nu t)$ be the operator $e^{\nu tB}(g\mathcal{H}^1_{\mathbb{T}\times \{y=0\}})$, where $\mathcal{H}^1_{\mathbb{T}\times \{y=0\}}$ is the one-dimensional Hausdorff measure restricted on the boundary. Then, $\partial_x$ commutes with both $e^{\nu t B}$ and $\Gamma(\nu t)$. In addition, for any $k\ge 0$, and for any $0\le s< t\le T$, there hold 
$$
\begin{aligned}
 ||| e^{\nu t B} f  |||_{\rho,\sigma,\delta(t),k}  \lesssim  ||| f|||_{\rho,\sigma, \delta(0),k}, & \qquad   ||| e^{\nu (t-s) B} f  |||_{\rho,\sigma,\delta(t),k}  \lesssim \sqrt{\frac ts} ||| f|||_{\rho,\sigma, \delta(s),k},
  \\ ||| \Gamma(\nu (t-s))g  |||_{\rho,\sigma,\delta(t),k} & \lesssim \sqrt{\frac t{t-s}} ||| g|||_{\rho,k} + \sqrt \nu  ||| g|||_{\rho,k+1} ,
   \end{aligned}$$
uniformly in the inviscid limit. Similarly, we also obtain 
$$
\begin{aligned}
 ||| e^{\nu t B} f  |||_{\mathcal{W}^{k,1}_{\rho,\sigma}}  \lesssim  ||| f|||_{\mathcal{W}^{k,1}_{\rho,\sigma}}, \qquad ||| \Gamma(\nu t)g  |||_{\mathcal{W}^{k,1}_{\rho,\sigma}} & \lesssim ||| g|||_{\rho,k} ,
   \end{aligned}$$
uniformly in the inviscid limit. 
\end{proposition}

\subsection{Duhamel principle}
We first treat the Stokes problem on $\TT \times \RR_+$. By  taking the Fourier transform in $x$, the problem is reduced to 
\beq\label{Stokes-a} 
\begin{aligned} 
\partial_t\omega_\alpha-\nu \Delta_\alpha\omega_\alpha&=f_\alpha(t,x,z), \qquad \mbox{in}\quad \RR_+\\
\nu (\partial_z + |\alpha|)\omega_\alpha &=g_\alpha(t), \qquad \qquad \mbox{on} \quad z=0,
\end{aligned}
\eeq 
in which $\omega_\alpha$ denotes the Fourier transform of $\omega$ with respect to $x$, and $\Delta_\alpha = \partial_z^2 - \alpha^2$. Let $G_\alpha(t,z;y)$ be the corresponding Green function of the linear Stokes problem \eqref{Stokes-a}. That is, for each fixed $y \ge 0$, the function $G_\alpha(t,z;y)$ solves 
\beq\label{Stokes-Gr} 
\begin{aligned} 
( \partial_t-\nu \Delta_\alpha) G_\alpha(t,z;y) &=0, \qquad \mbox{in}\quad \RR_+\\
\nu (\partial_z + |\alpha|) G_\alpha(t,z;y) &=0, \qquad \mbox{on} \quad z=0,
\end{aligned}
\eeq 
 together with the initial data $G_\alpha(0,z;y) = \delta_y(z)$. The Green function will be constructed so that $G_\alpha (t,\cdot;y)\in L^1$ for each $t,y$. It follows that 
 
 \begin{lemma}[Duhamel's principle]\label{lem-Duh} For any $T>0$, and for any $f_\alpha \in L^\infty(0,T; L^1(\RR_+)) $ and $g_\alpha \in L^\infty(0,T)$, the unique solution to the linear Stokes problem \eqref{Stokes-a}, with the initial data $\omega_\alpha(0,z) = \omega_{0,\alpha}(z)$ in $L^1(\RR_+)$, satisfies 
 \begin{equation}\label{Duh-Stokes}
\begin{aligned}
 \omega_\alpha(t,z) &= \int_0^\infty  G_\alpha(t,y;z) \omega_{0,\alpha} (y) \; dy + \int_0^t G_\alpha(t-s, 0; z) g_\alpha (s)\; ds
 \\&\quad+ \int_0^t\int_0^\infty G_\alpha(t-s,y;z) f_\alpha (s,y)\; dyds. 
\end{aligned} \end{equation}
 \end{lemma}
 
\begin{proof} Using \eqref{Stokes-a}, we compute 
$$
\begin{aligned} 
&
\int_0^t\int_0^\infty G_\alpha(t-s,y;z) f_\alpha (s,y)\; dyds 
\\& = 
\int_0^t\int_0^\infty  G_\alpha(t-s,y;z)  (\partial_s + \nu \alpha^2 - \nu \partial_y^2) \omega_\alpha (s,y) \; dyds 
\\& = 
\int_0^t\int_0^\infty (\partial_s + \nu \alpha^2 - \nu \partial_y^2)  G_\alpha(t-s,y;z)  \omega_\alpha (s,y) \; dyds + \int_0^\infty  G_\alpha(0,y;z) \omega_\alpha (t,y) \; dy 
\\&\quad -  \int_0^\infty  G_\alpha(t,y;z) \omega_{0,\alpha} (y) \; dy + \nu \int_0^t \Big( G_\alpha(t-s,y;z)  \partial_y \omega_\alpha - \partial_y G_\alpha(t-s,y;z) \omega_\alpha \Big)_{\vert_{y=0}}\; ds 
.\end{aligned}
$$
The lemma follows, upon using the initial data and boundary conditions on $G_\alpha(t,y;z)$. 
\end{proof}

\subsection{The Green function for the Stokes problem}
In this section, we derive sufficient pointwise bounds on the temporal Green function for the linear Stokes problem \eqref{Stokes-a}. Precisely, we prove the following. 

\begin{proposition}\label{prop-Stokes-Green} Let $G_{\alpha}(t,y;z)$ be the Green function of the Stokes problem \eqref{Stokes-a}. There holds 
\begin{equation}\label{Stokes-tG}
G_\alpha(t,y;z) = H_\alpha(t,y;z) + R_\alpha (t,y;z), 
\end{equation}
in which $H_\alpha(t,y;z) $ is exactly the one-dimensional heat kernel with the homogenous Neumann boundary condition and $R_\alpha(t,y;z)$ is the residual kernel due to the boundary condition. Precisely, 
There hold 
$$
\begin{aligned}
H_\alpha(t,y;z) & = \frac{1}{\sqrt{\nu t}} \Big( e^{-\frac{|y-z|^{2}}{4\nu t}} +  e^{-\frac{|y+z|^{2}}{4\nu t}} \Big) e^{-\alpha^{2}\nu t}, 
\\ | \partial_z^k R_\alpha (t,y;z)| &\lesssim \mu_f^{k+1} e^{-\theta_0\mu_f |y+z|} +  (\nu t)^{-\frac{k+1}{2}}e^{-\theta_0\frac{|y+z|^{2}}{\nu t}}  e^{-\frac18 \alpha^{2}\nu t} ,
\end{aligned}$$
for $y,z\ge 0$, $k\ge 0$, and for some $\theta_0>0$ and for $\mu_f = |\alpha| + \frac{1}{\sqrt \nu}$. 
\end{proposition}

\begin{remark}
{\em We note that the residual term $R_\alpha(t,y;z)$ contains a term without viscous dissipation $e^{-\alpha^2 \nu t}$, and this is precisely due to the $|\alpha|$ term in the boundary condition in the linear Stokes problem \eqref{Stokes-Gr}. Observe that $\omega_\alpha = \alpha e^{-\alpha z}$ is an exact stationary solution to the linear homogenous Stokes problem \eqref{Stokes-a}. 
}
\end{remark}

\begin{remark}
{\em By the reflection method (e.g., \cite{Maekawa2}), the residual Green kernel can be explicitly defined by 
\begin{equation}\label{Gr-Mae} R_\alpha(t,y;z) = 2e^{-\alpha^2 \nu t}(\alpha^2+\alpha \partial_z)(-\triangle_{\alpha}^{-1})G(\nu t,y+z),\end{equation}
with $G(t,z) = \frac{1}{\sqrt{4\pi t}}e^{-z^2/4t}.$ The pointwise bounds as derived in Proposition \ref{prop-Stokes-Green} are in particular useful in propagating unbounded vorticity with boundary layer behaviors.  
}
\end{remark}

We proceed the construction of the Green function via the resolvent equation. Namely, for each fixed $y\ge 0$, let $G_{\lambda,\alpha}(y,z)$ be the $L^1$ solution to the resolvent problem 
\beq\label{Stokes-res}
\bega 
( \lambda - \nu \Delta_\alpha )G_{\lambda,\alpha}(y,z)&=\delta_{y}(z)\\
\nu (\partial_z + |\alpha|) G_{\lambda,\alpha}(y,0) &=0.
\enda 
\eeq 
We then obtain the following. 
\begin{lemma} Let $\mu = \nu^{-1/2}\sqrt{\lambda + \alpha^2 \nu}$, having positive real part. There holds 
\beq \label{green}
G_{\lambda,\alpha}(y,z) = H_{\lambda,\alpha}(y,z) + R_{\lambda,\alpha}(y,z)
\eeq 
in which $H_{\lambda,\alpha}(y,z) $ denotes the resolvent kernel of the heat problem with homogenous Neumann boundary condition; namely, 
$$H_{\lambda,\alpha}(y,z) = \frac{1}{2\nu\mu}(e^{-\mu|y-z|}+e^{-\mu |y+z|}), \qquad R_{\lambda,\alpha}(y,z) = \frac{\alpha(\alpha+\mu)}{\lambda \mu }e^{-\mu|y+z|}.$$
In particular, $G_{\lambda,\alpha}(y,z)$ is meromorphic with respect to $\lambda$ in $\CC \setminus \{ -\alpha^2 \nu - \RR_+\}$ with a pole at $\lambda =0$. 
\end{lemma}
\begin{proof} The construction is standard, upon noting that $G_{\lambda,\alpha}(y,z)$ is a linear combination of $e^{\pm \mu z}$ and satisfies the following jump conditions across $z=y$:
$$ [G_{\lambda,\alpha}(y,z)]_{\vert_{z=y}} = 0, \qquad [\nu \partial_z G_{\lambda,\alpha}(y,z)]_{\vert_{z=y}} = 1.$$
The lemma follows. \end{proof}

\begin{proof}[Proof of Proposition \ref{prop-Stokes-Green}]
The temporal Green function $G_\alpha(t,z;y)$ can then be constructed via the inverse Laplace transform: 
\begin{equation}\label{Stokes-Green} G_{\alpha}(t,y;z)=\frac{1}{2\pi i}\int_\Gamma e^{\lambda t}G_{\lambda,\alpha}(y,z)d\lambda\end{equation}
in which the contour of integration $\Gamma$ is taken such that it remains on the right of the (say, $L^2$) spectrum of the linear operator $\lambda - \nu \Delta_\alpha$, which is $-\alpha^2 \nu - \RR_+$. 

In view of \eqref{green}, we set $H_{\alpha}(t,y;z)$ and $R_{\alpha}(t,y;z)$ to be the corresponding temporal Green function of $H_{\lambda,\alpha}(y,z)$
 and $R_{\lambda,\alpha}(y,z)$, respectively. It follows that $H_{\alpha}(t,y;z)$ is the temporal Green function of the one-dimensional heat problem with the homogenous Neumann boundary condition, yielding 
 $$H_{\alpha}(t,y;z) = \frac{1}{\sqrt{4\pi \nu t}} \Big( e^{-\frac{|y-z|^2}{4\nu t}} + e^{-\frac{|y+ z|^2}{4\nu t}} \Big) e^{-\nu \alpha^2 t}.$$
It remains to compute the residual Green function $R_\alpha(t,y;z)$:
\begin{equation}\label{def-Ra}R_{\alpha}(t,y;z)=\frac{1}{2\pi i}\int_\Gamma e^{\lambda t} e^{-\mu|y+z|} \frac{\alpha(\alpha+\mu)}{\lambda \mu } \; d\lambda .\end{equation}
Note that the $z$-derivative of $R_\alpha (t,y;z) $ gains an extra $\mu$ in the above integral. 

~\\
{\bf Case 1: $\alpha^2 \nu \le 1$.} By the Cauchy's theory, we may decompose the contour of integration as $\Gamma = \Gamma_\pm \cup \Gamma_c$, having 
\begin{equation}\label{def-Gc} \begin{aligned}
\Gamma_\pm : &= \Big\{ \lambda  =- \frac12\nu \alpha^2 + \nu (a^2 - b^2 ) + 2\nu ia b \pm iM , \quad \pm b \in \RR_+\Big\} ,
\\
\Gamma_c : &= \Big\{ \lambda  = - \frac12 \nu \alpha^2 + \nu a^2 + M e^{i\theta}, \quad \theta \in [-\pi/2, \pi/2]\Big\} ,
\end{aligned}
\end{equation}
for some positive number $M$ and for $a =\frac{|y+z|}{2\nu t} $. Since $\alpha^2\nu \le 1$, we take $M$ large enough so that the pole $\lambda=0$ remains on the left of the contour $\Gamma$. It is clear that $|\lambda |\gtrsim 1$ on $\Gamma$. 
 
On $\Gamma_c$, we note that 
$$
\begin{aligned}
\Re \mu &=  \nu^{-1/2} \Re \sqrt{\frac12 \nu \alpha^2 + \nu a^2 + M e^{i\theta}} \ge \nu^{-1/2} \sqrt{\frac12 \nu \alpha^2 + \nu a^2}  \ge a,
\\
\Re \mu &=  \nu^{-1/2} \Re \sqrt{\frac12 \nu \alpha^2 + \nu a^2 + M e^{i\theta}} \ge \nu^{-1/2} \sqrt{M}
.\end{aligned}$$
This implies that $\Re \mu \ge \frac a2 + \theta_0 \mu_f $ for some positive constant $\theta_0$, recalling $\mu_f = \alpha + 1/\sqrt \nu$  and $\alpha \le \nu^{-1/2}$. In particular, $|\mu| \gtrsim \mu_f \ge \alpha$.  This proves that 
$$ \begin{aligned}
\Big |\int_{\Gamma_c} e^{\lambda t} e^{-\mu|y+z|} \frac{\alpha(\alpha+\mu)}{\lambda \mu } \; d\lambda\Big| 
&\lesssim \int_{-\pi/2}^{\pi/2} e^{Mt} e^{ a^2 \nu t} e^{ - \frac a2 |y+z|} e^{-\theta_0 \mu_f|y+z| } \mu_f  d\theta  
\\
&\lesssim \mu_f e^{-\theta_0 \mu_f|y+z| } e^{ a^2 \nu t} e^{ - \frac a2 |y+z|} 
\\
&\lesssim \mu_f e^{-\theta_0 \mu_f|y+z| } ,
\end{aligned}$$ 
in which we used $e^{ a^2 \nu t} e^{ - \frac a2 |y+z|}   = 1$ by definition of $a$. As for derivatives, we estimate  
$$ \begin{aligned}
\Big |\int_{\Gamma_c} e^{\lambda t} e^{-\mu|y+z|} \frac{\alpha(\alpha+\mu)}{\lambda} \; d\lambda\Big| & = \Big |\int_{\Gamma_c} e^{\lambda t} e^{-\mu|y+z|} \frac{\alpha}{\nu (\mu - \alpha)} \; d\lambda\Big| 
\\&\lesssim \nu^{-1} \int_{-\pi/2}^{\pi/2} e^{Mt} e^{ a^2 \nu t} e^{ - \frac a2 |y+z|} e^{-\theta_0 \mu_f|y+z| } d\theta  
\\
&\lesssim \mu_f^2 e^{-\theta_0 \mu_f|y+z| } ,
\end{aligned}$$ 
upon recalling $\mu_f = \alpha + 1/\sqrt \nu$.

Meanwhile, on $\Gamma_\pm$, we note that  
$$
\begin{aligned}\Re\mu &= \Re \sqrt{\frac12 \alpha^2 + (a+ib)^2 \pm i \nu^{-1} M} \ge \Re \sqrt{(a+ib)^2} = a,\end{aligned}$$
upon noting that the sign of $b$ and $\pm M$ is the same on $\Gamma_\pm$. Similarly, we note that $\Re \mu \gtrsim M/\sqrt \nu$. 
By definition of $a$, we have 
$$ | e^{\lambda t} e^{-\mu|y+z|} |\le  e^{-\frac12\nu \alpha^2 t} e^{-\frac{|y+z|^2}{4\nu t}} e^{-\nu b^2 t},$$ 
and together with the fact that $\lambda = \nu (\mu^2 - \alpha^2)$, we compute 
$$\frac{\alpha(\alpha+\mu)}{\lambda \mu } \; d\lambda  =\frac{2i\alpha (a+ib) }{\mu (\mu - \alpha)} db.$$
Since $\alpha^2 \nu\lesssim 1$, we have $\alpha \lesssim |\mu|$. In addition, we have $(a+ib)^2 \pm i\nu^{-1/2}M = \mu^2 - \alpha^2$  on $\Gamma_\pm$ with $b$ having the same sign as does $\pm   M$. This implies that $|a+ib|^2 \le |\mu^2 - \alpha^2| \lesssim |\mu|^2$.  Putting the above computations together, we obtain 
$$
\begin{aligned}
\Big |\int_{\Gamma_\pm} e^{\lambda t} e^{-\mu|y+z|} \frac{\alpha(\alpha+\mu)}{\lambda \mu } \; d\lambda\Big|  
&\le  C_0 e^{-\frac12\nu \alpha^2 t} e^{-\frac{|y+z|^2}{4\nu t}} \int_\RR e^{-\nu b^2 t} \frac{\alpha |a+ib| }{|\mu (\mu - \alpha)|} db 
\\
&\le  C_0 e^{-\frac12\nu \alpha^2 t} e^{-\frac{|y+z|^2}{4\nu t}} \int_\RR e^{-\nu b^2 t} db 
\\&\le  C_0 (\nu t)^{-1/2}e^{-\frac12\nu \alpha^2 t} e^{-\frac{|y+z|^2}{4\nu t}} .\end{aligned} $$
As for derivatives, we estimate 
$$
\begin{aligned}
\Big |\int_{\Gamma_\pm} e^{\lambda t} e^{-\mu|y+z|} \frac{\alpha(\alpha+\mu)}{\lambda } \; d\lambda\Big|  
&\le  C_0 e^{-\frac12\nu \alpha^2 t} e^{-\frac{|y+z|^2}{4\nu t}} \int_\RR e^{-\nu b^2 t}  (a + |b|)db 
\\&\le  C_0 (\nu t)^{-1} \Big( 1 + \frac{|x+y|}{\sqrt{\nu t}}\Big) e^{-\frac12\nu \alpha^2 t} e^{-\frac{|y+z|^2}{4\nu t}} 
\\&\le  C_0 (\nu t)^{-1}e^{-\frac12\nu \alpha^2 t} e^{-\frac{|y+z|^2}{8\nu t}} 
.\end{aligned} $$

~\\
{\bf Case 2: $\alpha^2 \nu \ge 1$.} Take $a =\frac{|y+z|}{2\nu t} $ as in the previous case. Consider first the case when $|a-\alpha|\ge \frac12\alpha$. In this case, we move the contour of integration to  
$$\Gamma_1 := \Big\{ \lambda  =- \nu \alpha^2 + \nu (a^2 - b^2 ) + 2\nu ia b, \quad \pm b \in \RR_+\Big\} $$
which may pass the pole at $\lambda =0$ (precisely, it does when $a =\alpha$). 
By the Cauchy's theory, we have $$R_{\alpha}(t,y;z)=\frac{1}{2\pi i}\int_{\Gamma_1} e^{\lambda t} e^{-\mu|y+z|} \frac{\alpha(\alpha+\mu)}{\lambda \mu } \; d\lambda  + \mathrm{Res}_0$$
in which the residue at the pole $\lambda =0$ is given by \begin{equation}\label{residue} \mathrm{Res}_0 = 2\alpha e^{-\alpha |y+z|} \end{equation}
when $a<\alpha$. We take $ \mathrm{Res}_0 =0$ when $a>\alpha$. Note that the residue does not decay in time. This accounts for the contribution of the inhomogenous Neumann boundary condition. Since $\alpha^2 \nu \ge1$, we have $\mu_f = \alpha + 1/\sqrt \nu \le 2\alpha$, and hence 
$$ \mathrm{Res}_0 \le 2\mu_fe^{-\frac12\mu_f |y+z|}.$$

As for the integral term, we note that $\mu = a + ib$ and hence 
$$ \frac{\alpha(\alpha+\mu)}{\lambda \mu } \; d\lambda = \frac{2i \nu \alpha (\alpha + \mu)}{\lambda}db.$$
Note that $\Gamma_1$ cuts the real axis at $\nu (a^2 - \alpha^2)$ and the imaginary axis at $2\nu ab_0$ (when $a>\alpha$), with $b_0 = \pm\sqrt{a^2-\alpha^2}$.  This in particular yields $|\lambda|\gtrsim \nu \alpha (a+\alpha)$ and $|\mu| =\nu^{-1/2}|\sqrt{\lambda + \alpha^2 \nu}|\lesssim \nu^{-1/2}|\lambda|^{1/2}. $ In particular, 
$$ \Big|\frac{\alpha(\alpha+\mu)}{\lambda \mu } \; d\lambda \Big| \le db.$$
We thus obtain 
\begin{equation}\label{est-G1}
\begin{aligned}
\Big |\int_{\Gamma_1} e^{\lambda t} e^{-\mu|y+z|} \frac{\alpha(\alpha+\mu)}{\lambda \mu } \; d\lambda\Big|  
&\le  C_0 e^{-\nu \alpha^2 t} e^{-\frac{|y+z|^2}{4\nu t}} \int_\RR e^{-\nu b^2 t} db 
\\&\le  C_0 (\nu t)^{-1/2}e^{-\nu \alpha^2 t} e^{-\frac{|y+z|^2}{4\nu t}} .\end{aligned} \end{equation}

It remains to consider the case when $|a-\alpha| \le \frac 12\alpha$ and $\alpha^2 \nu \ge 1$. We note in particular that $\frac 12 \alpha \le a\le \frac 32\alpha$. In this case, we simply modify the contour of integration as follows: we take 
$$\Gamma_1 := \Big\{ \lambda  =- \frac18\nu \alpha^2 + \nu (a^2 - b^2 ) + 2\nu ia b, \quad \pm b \in \RR_+\Big\} .$$
Observe that the contour $\Gamma_1$ leaves the origin on the left, with $|\lambda|\gtrsim \nu \alpha^2$. The integral is thus estimated exactly as done in \eqref{est-G1}.  The derivative estimates follow as in the previous case. 

The proof of Proposition \ref{prop-Stokes-Green} is complete. 
\end{proof}

\subsection{The Green function on $\Omega_{\sigma}$}\label{sec-Grcomplex}
The Green function constructed in Proposition \ref{prop-Stokes-Green} can be easily extended to the complex domain $\Omega_{\sigma}$ defined by 
$$
\Omega_{\sigma}= \Big\{z\in\mathbb{C}:\quad|\Im z|<\min\{\sigma|\Re z|,\sigma \}\Big\},
$$
for some small $\sigma>0$. Indeed, in view of \eqref{Gr-Mae}, the Green function involves precisely the heat kernel $G(t,z) = \frac{1}{\sqrt{4\pi t}} e^{-z^2 / 4t}$, which is extended to the complex domain. In addition, we note that for $z\in \Gamma_{\sigma}$, there holds $\Im z \le \sigma \Re z$, which implies that 
$$ |e^{-z^2 / 4t}| \le e^{- |\Re z|^2/4t + |\Im z|^2 / 4t} \le e^{- (1-\sigma^2) |\Re z|^2 / 4t}.$$
Similar estimates hold for the other terms in the Green function $G_{\alpha}(t,y;z) =  H_\alpha(t,y;z) + R_\alpha (t,y;z)$, yielding \begin{equation}\label{Stokes-Gr-complex}
\begin{aligned}
H_\alpha(t,y;z) &\lesssim \frac{1}{\sqrt{\nu t}} \Big( e^{-(1-\sigma^2)\frac{|\Re y- \Re z|^{2}}{4\nu t}} +  e^{-(1-\sigma^2)\frac{|\Re y+\Re z|^{2}}{4\nu t}} \Big) e^{-\frac{1}{8}\alpha^{2}\nu t}, 
\\R_\alpha (t,y;z) &\lesssim \mu_f e^{-\theta_0(1-\sigma)\mu_f |\Re y+\Re z|} ,
\end{aligned}\end{equation}
for $y,z \in \Gamma_{\sigma}$, and for some $\theta_0>0$ and for $\mu_f = |\alpha| + \frac{1}{\sqrt \nu}$. 

The solution $\omega_\alpha(t,z)$ to the Stokes problem can now be constructed on $\Omega_\sigma$ in the similar manner as done in \eqref{laplace-3}. Precisely, for any $z\in \Omega_\sigma$, let $\theta$ be the positive constant so that $z\in \partial \Omega_\theta$. The Duhamel principle \eqref{Duh-Stokes} then becomes 
 \begin{equation}\label{Duh-Stokes-a}
\begin{aligned}
 \omega_\alpha(t,z) &= \int_{\partial\Omega_\theta}G_\alpha(t,y;z) \omega_{0,\alpha} (y) \; dy + \int_0^t G_\alpha(t-s, 0; z) g_\alpha (s)\; ds
 \\&\quad+ \int_0^t  \int_{\partial\Omega_\theta} G_\alpha(t-s,y;z) f_\alpha (s,y)\; dyds,
\end{aligned} \end{equation}
which is well-defined for $z\in \Omega_\sigma$, having the Green function $G_\alpha(t,y;z) $ satisfies the pointwise estimates \eqref{Stokes-Gr-complex}, similar to those on the real line. For this reason, it suffices to derive convolution estimates for real values $y,z$.

\subsection{Convolution estimates}\label{sec-convS}

We now derive convolution estimates. We start with the analytic $L^1$ norms. For $k\ge 0$, we introduce   
$$ \| \omega_\alpha\|_{\mathcal{W}^{k,1}_{\sigma}} = \sum_{j=0}^k \|(\psi(z)\partial_z)^j \omega_\alpha \|_{L^1_\sigma} .$$ 
We prove the following.

\begin{proposition}\label{prop-Stokes-convL1} Let $T>0$ and let $G_{\alpha}(t,y;z)$ be the Green function of the Stokes problem \eqref{Stokes-a}, constructed in Proposition \ref{prop-Stokes-Green}. 
Then, for any $0\le s <  t\le T$ and $k\ge 0$, there is a universal constant $C_T$ so that 
$$
\begin{aligned}
\Big \| \int_0^\infty G_\alpha(t,y;\cdot) \omega_\alpha(y)\; dy \Big\|_{\mathcal{W}^{k,1}_{\sigma}} &\le C_T   \| \omega_\alpha\|_{\mathcal{W}^{k,1}_{\sigma}},
\\
\Big \| \int_0^\infty G_\alpha(t-s,y;\cdot) \omega_\alpha(y)\; dy \Big\|_{\mathcal{W}^{k,1}_{\sigma}} &\le C_T  \| \omega_\alpha\|_{\mathcal{W}^{k,1}_{\sigma}},
\end{aligned}$$
uniformly in the inviscid limit.  
\end{proposition}
\begin{proof} We shall prove the convolution for real values $y,z$. For the complex extension, see Section \ref{sec-Grcomplex}. 
Recall from  Proposition \ref{prop-Stokes-Green} that $
G_\alpha(t,y;z) = H_\alpha(t,y;z) + R_\alpha (t,y;z), 
$
with $$
\begin{aligned}
H_\alpha(t,y;z) & = \frac{1}{\sqrt{\nu t}} \Big( e^{-\frac{|y-z|^{2}}{4\nu t}} +  e^{-\frac{|y+z|^{2}}{4\nu t}} \Big) e^{-\alpha^{2}\nu t}, 
\\ |\partial_z^k R_\alpha (t,y;z)| &\lesssim \mu^{k+1}_f e^{-\theta_0\mu_f |y+z|} + (\nu t)^{-\frac{k+1}{2}} e^{-\theta_0\frac{|y+z|^{2}}{\nu t}}  e^{-\frac18 \alpha^{2}\nu t} 
\end{aligned}$$
for $k\ge 0$. In particular, $\|G_\alpha(t,y;\cdot)\|_{L^1_\sigma} \lesssim 1$, for each fixed $y,t$. The $L^1$ convolution estimate is thus straightforward. We now check the estimates for derivatives. We estimate $$\begin{aligned}
 \psi(z)\partial_z R_\alpha (t,y;z) &\lesssim |z| \mu^2_f e^{-\theta_0\mu_f |y+z|} +  \frac{|z|}{\nu t}  e^{-\theta_0\frac{|y+z|^{2}}{\nu t}}  e^{-\frac18 \alpha^{2}\nu t}  
 \\
 &\lesssim \mu_f e^{-\frac12 \theta_0\mu_f |y+z|} +  \frac{1}{\sqrt{\nu t}}  e^{-\theta_0\frac{|y+z|^{2}}{2\nu t}}  e^{-\frac18 \alpha^{2}\nu t}  .
 \end{aligned}$$
That is, $ \psi(z)\partial_z R_\alpha (t,y;z) $ obeys essentially the same bound as does $R_\alpha(t,y;z)$. The convolution estimates for derivatives of $R_\alpha(t,y;z)$ follow. 

Next, we treat the integral involving $H_\alpha(t,y;z)$. Precisely, we set $H(t,y;z) = \frac{1}{\sqrt{\nu t}} e^{-\frac{|y-z|^{2}}{4\nu t}} $. 
Note that $\partial_z H(t,y;z) = - \partial_y H(t,y;z)$. Hence, we compute 
$$\begin{aligned}
&\int_0^\infty \psi(z)\partial_zH(t-s,y;\cdot) \omega_\alpha(y)\; dy 
\\&= \int_0^{z/2} \psi(z)\partial_zH(t-s,y;\cdot) \omega_\alpha(y)\; dy
- \psi(z) H(t-s,z/2;z) \omega_\alpha(z/2)
\\&\quad + \int_{z/2}^\infty \psi(z)H(t-s,y;\cdot) \partial_y\omega_\alpha(y)\; dy  .
\end{aligned}$$
We now estimate each term on the right. Since $\psi(z) \le 2 \psi(y)$ for $y\ge z/2$, the last integral on the right is already estimated in the previous case with $\omega_\alpha(y)$ replaced by $\psi(y)\partial_y \omega_\alpha(y)$. As for the first integral, since $y\le z/2$, we compute  
$$
\begin{aligned}\psi(z)\partial_z H(t-s,y;z) &\lesssim \frac{z}{1+z}(\nu (t-s))^{-1} e^{-\frac{|y-z|^2}{8\nu (t-s)}} 
\\&\lesssim |y-z|(\nu (t-s))^{-1} e^{-\frac{|y-z|^2}{8\nu (t-s)}}
\\& \lesssim (\nu (t-s))^{-1/2} e^{-\frac{|y-z|^2}{16\nu (t-s)}}.
\end{aligned}$$ 
Thus, the integral over $[0,z/2]$ is again already estimated in the previous case. Finally, we compute 
$$
\begin{aligned}
| \psi(z) H(t-s,z/2;z) \omega_\alpha(z/2)| 
&\lesssim z (\nu (t-s))^{-1/2} e^{-\frac{|z|^2}{16\nu (t-s)}}  |\omega_\alpha(z/2)|  \lesssim  |\omega_\alpha(z/2)| ,
\end{aligned}$$
whose $L^1_\sigma$ norm is clearly bounded by $\|\omega_\alpha \|_{L^1_\sigma}$. 
\end{proof}

\subsection{Convolution estimates with boundary layer behaviors}
In this section, we provide the convolution estimates of the Green function against functions in the boundary layer spaces, whose norm is defined by 
\begin{equation}\label{def-blnormFull123}
\| \omega_\alpha\|_{\sigma,\delta(t)}  = \sup_{z\in \Omega_\sigma} | \omega_\alpha(z) | e^{\beta \Re z}
\Bigl( 1 + \delta_t^{-1} \phi_{P} (\delta_t^{-1} z)  +  \delta^{-1} \phi_{P} (\delta^{-1} z)  \Bigr)^{-1} ,
\end{equation}
for $t>0$ and $\beta>0$, in which the boundary thicknesses are $\delta_t = \sqrt{\nu t}$ and $\delta = \sqrt \nu $ and for boundary layer weight $\phi_P(z) = \frac{1}{1+|\Re z|^P}$, $P>1$. We also introduce the boundary norm for derivatives: 
$$ \| \omega_\alpha\|_{\sigma,\delta(t),k} = \sum_{j=0}^k \|(\psi(z)\partial_z)^j \omega_\alpha \|_{\sigma,\delta(t)} $$ 
for $k\ge 0$. In the case $t=0$, the norm $\|\cdot \|_{\sigma,\delta(0)}$ is defined to consist of precisely one boundary layer with thickness $\delta = \sqrt \nu$.

We prove the following.

\begin{proposition}\label{prop-Stokes-conv} Let $T>0$ and let $G_{\alpha}(t,y;z)$ be the Green function of the Stokes problem \eqref{Stokes-a}, constructed in Proposition \ref{prop-Stokes-Green}. 
Then, for any $0\le s <  t\le T$ and $k\ge 0$, there is a universal constant $C_T$ so that 
$$
\begin{aligned}
\Big \| \int_0^\infty G_\alpha(t,y;\cdot) \omega_\alpha(y)\; dy \Big\|_{\sigma, \delta(t),k} &\le C_T   \| \omega_\alpha\|_{\sigma, \delta(0),k},
\\
\Big \| \int_0^\infty G_\alpha(t-s,y;\cdot) \omega_\alpha(y)\; dy \Big\|_{\sigma, \delta(t),k} &\le C_T \sqrt{\frac ts}  \| \omega_\alpha\|_{\sigma, \delta(s),k},
\end{aligned}$$
uniformly in the inviscid limit.  
\end{proposition}

We shall prove the convolution estimates for real values $y,z$. The complex extension follows from the similar estimates on the Green function obtained in \eqref{Stokes-Gr-complex}. As a consequence, Proposition \ref{prop-Stokes-conv} is a direct combination of the following two lemmas. 

\begin{lemma}\label{lem-Lapconv} Let $R(t,y;z): =  \mu_f e^{-\mu_f |y+z|} $, with $\mu_f =  \alpha + \frac{1}{\sqrt\nu}$. Then, for any $s,t,$ and $k\ge 0$, there is a universal constant $C_0$ so that 
$$
\begin{aligned}
\Big\| \int_0^\infty R(t-s,y;z) \omega_\alpha(y)\; dy \Big\|_{\sigma,\delta(t),k}\le C_0 \| \omega_\alpha\|_{\sigma,\delta(s)} .
\end{aligned}$$
 \end{lemma}
\begin{proof} The estimate for $k=0$ follows directly from the fact that $\omega_\alpha$ belongs to $L^1$: $\|\omega_\alpha\|_{L^1} \le C_0 \| \omega_\alpha\|_{\sigma,\delta(s)}$. The convolution in fact belongs to the boundary layer space with finite norm $\|\cdot \|_{\sigma,\delta(0)}$. As for derivatives, we note that $| \psi(z)\partial_z R(t,y;z)| \le z  \mu_f^2 e^{-\mu_f |y+z|} \le C_0 \mu_f e^{-\frac12 \mu_f z}$. The derivative estimates thus follow identically. \end{proof}

\begin{lemma}\label{lem-Heatconv} Let $H(t,y;z): = (\nu t)^{-1/2} e^{-\frac{|y\pm z|^2}{M \nu t}} $, for some positive $M$, and let $T>0$. Then, for any $0\le s <  t\le T$, $\epsilon_0>0$, and $k\ge 0$, there is a universal constant $C_T$ so that 
$$
\begin{aligned}
\Big \| \int_0^\infty H(t,y;\cdot) \omega_\alpha(y)\; dy \Big\|_{\sigma, \delta(t),k} &\le C_T  \| \omega_\alpha\|_{\sigma, \delta(0),k},
\\
\Big \| \int_0^\infty H(t-s,y;\cdot) \omega_\alpha(y)\; dy \Big\|_{\sigma, \delta(t),k} &\le C_T\sqrt{\frac{t}{s}}\| \omega_\alpha\|_{\sigma, \delta(s),k},
\end{aligned}$$
uniformly in the inviscid limit. 
\end{lemma}
\begin{proof} It suffices to prove the convolution for $H(t,y;z) =   (\nu t)^{-1/2}  e^{-\frac{|y -  z|^2}{M \nu t}}$. We start with the case $k=0$. Let $0\le s <  t$. For $|y-z|\ge M\beta   \nu (t-s)$, it is clear that 
$$e^{-\frac{|y-z|^2}{M \nu (t-s)}} e^{-\beta  |y|} \le e^{-\beta  |z|} e^{-|y-z| \Big( \frac{|y-z|}{M \nu (t-s)} - \beta \Big)} \le e^{-\beta  |z|}.$$
Whereas, for $|y-z| \le M \beta  \nu (t-s)$, we note that 
$$ e^{- M \beta^2  \nu (t-s)} e^{-\beta  |y|} \le e^{-\beta |y-z|} e^{-\beta  |y|} \le e^{-\beta  |z|}.$$
That is, the exponential decay $e^{-\beta z}$ is recovered at an expense of a slowly growing term in time: $e^{M \beta^2\nu (t-s)}$, which is bounded in finite time. Precisely, this proves 
\begin{equation}\label{exp-beta13}e^{-\frac{|y-z|^2}{M \nu (t-s)}} e^{-\beta y} \le e^{M \beta^2 \nu (t-s)} e^{-\beta |z|}, \qquad \forall y,z\in \RR.\end{equation}

It remains to study the integral 
\begin{equation}\label{conv-heatbl1-stable}
\begin{aligned}
\int_0^\infty (\nu (t-s))^{-1/2}  e^{-\frac{|y-z|^2}{M \nu (t-s)}} 
 \Bigl( 1 + \delta_s^{-1} \phi_{P} (\delta_s^{-1} y)  +  \delta^{-1} \phi_{P} (\delta^{-1} y)   \Bigr) \; dy.
\end{aligned}\end{equation}
First, without the boundary layer behavior, the integral is clearly bounded. We now treat the boundary layer terms. Using the fact that $\phi_{P}(\cdot)$ is decreasing, we have 
$$
\begin{aligned}
 \int_{z/2}^\infty & (\nu (t-s))^{-1/2}   e^{-\frac{|y-z|^2}{M \nu (t-s)}} 
  \delta^{-1} \phi_{P} (\delta^{-1} y) \; dy 
  \\&\le C_0   \delta^{-1} \phi_{P} (\delta^{-1} z)  \int_{z/2}^\infty (\nu (t-s))^{-1/2}   e^{-\frac{|y-z|^2}{M \nu (t-s)}}  \; dy
  \\&\le C_0   \delta^{-1} \phi_{P} (\delta^{-1} z)  \end{aligned}$$
and, upon noting that $y/\delta_s \ge z/2\delta_t$, we obtain  
$$
\begin{aligned}
 \int_{z/2}^\infty & (\nu (t-s))^{-1/2}   e^{-\frac{|y-z|^2}{M \nu (t-s)}} 
  \delta_s^{-1} \phi_{P} (\delta_s^{-1} y) \; dy 
  \\&\le C_0   \delta_s^{-1} \phi_{P} (\delta_t^{-1} z)  \int_{z/2}^\infty (\nu (t-s))^{-1/2}   e^{-\frac{|y-z|^2}{M \nu (t-s)}}  \; dy
  \\&\le C_0 \sqrt{t/s} \delta_t^{-1} \phi_{P} (\delta_t^{-1} z) .  \end{aligned}$$
Whereas on $y\in (0,\frac z2)$, we have $|y-z|\ge \frac z2 $ and $\phi_{P} \le 1$. Hence, we have 
$$
\begin{aligned}
 \int_0^{z/2} &(\nu (t-s))^{-1/2}   e^{-\frac{|y-z|^2}{M \nu (t-s)}} 
  \delta_s^{-1} \phi_{P} (\delta_s^{-1} y) \; dy 
\\&\le  e^{-\frac{|z|^2}{8M \nu (t-s)}}\int_0^{z/2} (\nu (t-s))^{-1/2}   e^{-\frac{|y-z|^2}{2M \nu (t-s)}} 
  \delta_s^{-1} \phi_{P} (\delta_s^{-1} y) \; dy 
\\&\le C_0 e^{-\frac{|z|^2}{8M \nu t}}  \min\{ \delta_s^{-1}, \delta_{t-s}^{-1}\} .
  \end{aligned}
  $$
Note that $\min\{ \delta_s^{-1}, \delta_{t-s}^{-1}\} \le 2\delta_t^{-1}$. Hence, the above integral is bounded in $\|\cdot \|_{\sigma, \delta(t)}$ norm. Similarly, we estimate 
$$
\begin{aligned}
 \int_0^{z/2} &(\nu (t-s))^{-1/2}   e^{-\frac{|y-z|^2}{M \nu (t-s)}} 
  \delta^{-1} \phi_{P} (\delta^{-1} y) \; dy 
\le C_0 e^{-\frac{|z|^2}{8M \nu (t-s)}}  \delta^{-1}
.\end{aligned}$$
To estimate this, we will prove that \begin{equation}\label{heat-bl}
 e^{-\frac{|z|^2}{8M \nu t}}  \le C_0 e^{\epsilon_0 t}\phi_{P} (\delta^{-1} z) ,
\end{equation}
for arbitrarily small $\epsilon_0$ (and hence, $C_0$ depends on $\epsilon_0$). Indeed, when $|z| \ge \epsilon_0  \sqrt \nu t $, it is clear that $e^{-\frac{|z|^2}{8M \nu t}} \lesssim \phi_{P} (\delta^{-1} z)$. On the other hand, when $|z|\le \epsilon_0 \sqrt \nu t $, we note that $ e^{z/\delta} \le e^{ \epsilon_0 t}$, which implies that $1\le e^{\epsilon_0 t} e^{-z/\delta} \le C_0e^{\epsilon_0 t} \phi_{P} (\delta^{-1} z) $. The estimate \eqref{heat-bl} follows, and hence the claimed estimate for $k=0$. 

Next, we consider the derivative estimate. Note that $\partial_z H(t,y;z) = - \partial_y H(t,y;z)$. Hence, we compute 
$$ \begin{aligned}
&\int_0^\infty \psi(z)\partial_zH(t-s,y;\cdot) \omega_\alpha(y)\; dy 
\\&= \int_0^{z/2} \psi(z)\partial_zH(t-s,y;\cdot) \omega_\alpha(y)\; dy
- \psi(z) H(t-s,z/2;z) \omega_\alpha(z/2)
\\&\quad + \int_{z/2}^\infty \psi(z)H(t-s,y;\cdot) \partial_y\omega_\alpha(y)\; dy  .
\end{aligned}$$
We now estimate each term on the right. Since $\psi(z) \le 2 \psi(y)$ for $y\ge z/2$, the last integral on the right is already estimated in the previous case with $\omega_\alpha(y)$ replaced by $\psi(y)\partial_y \omega_\alpha(y)$. As for the first integral, since $y\le z/2$, we compute  
$$
\begin{aligned}\psi(z)\partial_z H(t-s,y;z) &\lesssim \frac{z}{1+z}(\nu (t-s))^{-1} e^{-\frac{|y-z|^2}{2M\nu (t-s)}} 
\\&\lesssim |y-z|(\nu (t-s))^{-1} e^{-\frac{|y-z|^2}{2M\nu (t-s)}}
\\& \lesssim (\nu (t-s))^{-1/2} e^{-\frac{|y-z|^2}{4M\nu (t-s)}}.
\end{aligned}$$ 
Thus, the integral over $[0,z/2]$ is again already estimated in the previous case. Finally, we compute 
$$
\begin{aligned}
&| \psi(z) H(t-s,z/2;z) \omega_\alpha(z/2)| 
\\&\lesssim z (\nu (t-s))^{-1/2} e^{-\frac{|z|^2}{4M\nu (t-s)}} \Big( 1+ \delta_s^{-1} \phi_{P} (\delta_s^{-1} z)+ \delta^{-1} \phi_{P} (\delta^{-1} z)\Big) 
\\
&\lesssim 1+ \delta_s^{-1} \phi_{P} (\delta_s^{-1} z)+ \delta^{-1} \phi_{P} (\delta^{-1} z)
\\
&\lesssim 1+ \sqrt{t/s}\delta_t^{-1} \phi_{P} (\delta_t^{-1} z)+ \delta^{-1} \phi_{P} (\delta^{-1} z),
\end{aligned}$$
in which again the last inequality is due to the decreasing property of $\phi_P$ and the fact that $\delta_s\le \delta_t$. 
This completes the proof of the lemma. 
\end{proof}

\subsection{Semigroup bounds in analytic spaces}

In this section, we shall prove Proposition \ref{prop-Stokes} on deriving uniform bounds for the Stokes semigroup in analytic spaces, with the analytic norm 
$$\| \omega\|  _{\rho,\sigma,\delta(t)}=\sum_{\alpha\in\mathbb{Z}} e^{\rho|\alpha|} \|\omega_{\alpha}\|_{\sigma,\delta(t)}$$
We first write the Stokes semigroup $e^{t B}$ and the operator $\Gamma(t)$ in the Fourier series: 
\begin{equation}\label{semigroup-exp} 
 e^{\nu t B} \omega = \sum_{\alpha \in \ZZ} e^{i\alpha x} (e^{\nu t B} \omega)_\alpha, \qquad \Gamma(\nu t) g = \sum_{\alpha \in \ZZ} e^{i\alpha x} (\Gamma(\nu t)g)_\alpha \end{equation}
in which 
\begin{equation}\label{semigroup-exp1} (e^{\nu t B} \omega)_\alpha  =  \int_0^\infty  G_\alpha(t,y;z) \omega_{\alpha} (y) \; dy, \qquad  (\Gamma(\nu t)g)_\alpha = G_\alpha(t, 0; z) g_\alpha,\end{equation}
with the Green kernel $G_\alpha(t,y;z)$ constructed in Proposition \ref{prop-Stokes-Green}. The convolution estimates obtained in Proposition \ref{prop-Stokes-conv} yield 
$$
\begin{aligned}
 \|  (e^{\nu (t-s) B} \omega)_\alpha \|_{\sigma, \delta(t),k} &\le C_0\sqrt{\frac ts} \| \omega_\alpha\|_{ \sigma,\delta(s),k}.
\end{aligned}$$
These prove the claimed estimates on $e^{\nu t B}$. 
As for the trace operator, using Proposition \ref{prop-Stokes-Green}, we have 
\begin{equation}\label{Green-bdry}
 \begin{aligned}
  G_\alpha(t-s, 0; z) 
&  \lesssim (\nu (t-s))^{-1/2} e^{-\frac{z^2}{ 4\nu (t-s)}} + \mu_f e^{-\mu_f z}
\\
&  \lesssim (\nu (t-s))^{-1/2} e^{-\frac{z^2}{ 4\nu t}}  + (|\alpha| + \nu^{-1/2}) e^{-\frac{z}{\sqrt\nu} },
 \end{aligned}\end{equation}
upon recalling that $\mu_f = |\alpha| + \nu^{-1/2}$. By definition, $\|   G_\alpha(t-s, 0; z)  \|_{\sigma,\delta(t)} \lesssim \sqrt{\frac{t}{t-s}} + 1 + \alpha \sqrt \nu.$
The proof of Proposition \ref{prop-Stokes}  is complete.

\section{Proof of the main theorems}\label{sec-proof}
As mentioned in the introduction, we construct the solutions to the Navier-Stokes equation via the vorticity formulation: 
\beq\label{NS-vor1}
\begin{aligned}
\partial_{t}\omega-\nu\Delta\omega &=-u\cdot\nabla\omega
\\
\nu (\partial_z + |\partial_x|)\omega_{\vert_{z=0}} &= [\partial_z \Delta^{-1} ( u \cdot \nabla \omega)]
_{\vert_{z=0}},
\end{aligned}\eeq
in which $u = \nabla^\perp \Delta^{-1} \omega$, with $\Delta^{-1}$ being the inverse of Laplacian with the Dirichlet boundary condition. For convenience, we set $N = u \cdot \nabla \omega$.  
The solution to \eqref{NS-vor1} is then constructed via the Duhamel's principle:
 \beq\label{Duh-w1}
\omega(t)=e^{\nu tB}\omega_{0} -\int_{0}^{t}e^{\nu(t-s)B}  N(s) \; ds + \int_0^t \Gamma(\nu (t-s)) (\partial_z \Delta^{-1} N(s))_{\vert_{z=0}} \;ds
\eeq 
with $\omega_0\in {\cal B}^{\rho_0,\sigma_0,\delta}$, for some $\rho_0,\sigma_0>0$. 

\subsection{Nonlinear iteration}\label{sec-nonlinear}

Let us fix positive numbers $\gamma, \zeta,$ and $\rho_0$, and introduce the following nonlinear iterative norm for vorticity: 
\beq\label{def-normw}
\bega 
A(\gamma)=&\quad \sup_{0<\gamma t< \rho_0}\sup_{\rho<\rho_0- \gamma t}\Bigl{\{} 
 ||| \omega(t)|||_{\mathcal{W}^{1,1}_{\rho,\rho}} +  ||| \omega(t)|||_{\mathcal{W}^{2,1}_{\rho,\rho}}(\rho_0-\rho-\gamma t)^{\zeta}\Bigr{\}}\\
\enda 
\eeq 
with recalling $$ ||| \omega(t)|||_{\mathcal{W}^{k,1}_{\rho,\rho}}  = \sum_{j+\ell \le k}  \|\partial_x^j (\psi(z)\partial_z)^\ell \omega(t)\|_{L^1_{\rho,\rho}}.$$
Here, for sake of presentation, we take the same analyticity radius in $x$ and $z$; namely, $\sigma = \rho <\rho_0$. Thanks to Lemma \ref{lem-emL1}, $\omega_0 \in \mathcal{W}^{k,1}_{\rho,\rho}$, for any $k\ge 0$. 

We shall show that the vorticity norm remains finite for sufficiently large $\gamma$. The weight $(\rho_0-\rho-\gamma t)^\zeta$, with a small $\zeta>0$, is standard to avoid time singularity when recovering the loss of derivatives (\cite{Caflisch,Safanov}). Let $\rho < \rho_0 - \gamma t$. Thanks to Lemma \ref{lem-bilinear}, we have \begin{equation}\label{non-est}
\begin{aligned}
 ||| N(t)|||_{\mathcal{W}^{0,1}_{\rho,\rho}} &\lesssim  ||| \omega(t)|||_{\mathcal{W}^{1,1}_{\rho,\rho}}^2 \le A(\gamma)^2
 \\
 ||| N(t)|||_{\mathcal{W}^{1,1}_{\rho,\rho}} &\lesssim  ||| \omega(t)|||_{\mathcal{W}^{1,1}_{\rho,\rho}}  ||| \omega(t)|||_{\mathcal{W}^{2,1}_{\rho,\rho}}  \le A(\gamma)^2 (\rho_0-\rho-\gamma t)^{-\zeta}
. \end{aligned}\end{equation}
In addition, using the elliptic estimates, we have 
\begin{equation}\label{non-est}
\begin{aligned}
 ||| (\partial_z \Delta^{-1} N(t))_{\vert_{z=0}} |||_{\rho,k} &\lesssim  ||| N(t)|||_{\mathcal{W}^{k,1}_{\rho,\rho}}
 . \end{aligned}\end{equation}
Now, using the Duhamel integral formula \eqref{Duh-w1}, we estimate 
$$
\begin{aligned}
||| \omega(t) |||_{\mathcal{W}^{k,1}_{\rho,\rho}} &\le ||| e^{\nu tB}\omega_{0}|||_{\mathcal{W}^{k,1}_{\rho,\rho}} +\int_{0}^{t} ||| e^{\nu(t-s)B} N (s)|||_{\mathcal{W}^{k,1}_{\rho,\rho}} \; ds 
\\&\quad + \int_0^t ||| \Gamma(\nu (t-s))(\partial_z \Delta^{-1} N(s))_{\vert_{z=0}}|||_{\mathcal{W}^{k,1}_{\rho,\rho}} \;ds.
\end{aligned}
$$
In view of Proposition \ref{prop-Stokes}, the term from the initial data is already estimated, giving $||| e^{\nu tB}\omega_{0}|||_{\mathcal{W}^{k,1}_{\rho,\rho}} \le \|\omega_0\|_{\mathcal{W}^{k,1}_{\rho,\rho}}$. As for the integral terms, we estimate 
$$\begin{aligned}
\int_{0}^{t} ||| e^{\nu(t-s)B} N(s)|||_{\mathcal{W}^{1,1}_{\rho,\rho}} \; ds 
&\le C_0 \int_0^t ||| N(s)|||_{\mathcal{W}^{1,1}_{\rho,\rho}} \; ds
\\&\le C_0 A(\gamma)^2 \int_0^t  (\rho_0-\rho-\gamma s)^{-\zeta}\; ds
\\&\le 
C_0 \gamma^{-1} A(\gamma)^2 .
\end{aligned}$$
Similarly, 
$$\begin{aligned}
\int_0^{t} ||| \Gamma(\nu (t-s))(\partial_z \Delta^{-1} N(s))_{\vert_{z=0}}|||_{\mathcal{W}^{1,1}_{\rho,\rho}} \;ds
&\le C_0 \int_0^t ||| N(s)|||_{\mathcal{W}^{1,1}_{\rho,\rho}} \; ds ,
\end{aligned}$$
which is again bounded by $C_0 \gamma^{-1} A(\gamma)^2$. Next, we give estimates for $k=2$. Noting that $\rho < \rho_0 - \gamma t \le \rho_0 - \gamma s$, we take $\rho' = \frac{\rho + \rho_0 - \gamma s}{2}$ and compute 
$$\begin{aligned}
\int_{0}^{t} ||| e^{\nu(t-s)B} N(s)|||_{\mathcal{W}^{2,1}_{\rho,\rho}} \; ds 
&\le C_0 \int_0^t ||| N(s)|||_{\mathcal{W}^{2,1}_{\rho,\rho}} \; ds 
\\&\le C_0 \int_0^t\frac{1}{\rho' - \rho}||| N(s)|||_{\mathcal{W}^{1,1}_{\rho',\rho'}}\; ds
\\&\le C_0 A(\gamma)^2 \int_0^t  (\rho_0-\rho-\gamma s)^{-1-\zeta}\; ds
\\&\le C_0 \gamma^{-1} A(\gamma)^2 (\rho_0-\rho-\gamma t)^{-\zeta}.
\end{aligned}$$
Same computation holds for the trace operator $\Gamma(\nu t)$, yielding 
$$ A(\gamma ) \le  C_0 \|\omega_0\|_{\mathcal{W}^{2,1}_{\rho,\rho}}  + C_0 \gamma^{-1} A(\gamma)^2 . $$ 
By taking $\gamma$ sufficiently large, the above yields the uniform bound on the iterative norm in term of initial data. This yields the local solution in $L^1_{\rho,\rho}$ for $t \in [0,T]$, with $T = \gamma^{-1}\rho_0$. 

\subsection{Propagation of boundary layers}
It remains to prove that the constructed solution has the boundary layer behavior as expected, having already constructed solutions in $L^1_{\rho,\rho}$ spaces. Indeed, we now introduce the following nonlinear iterative norm for vorticity: 
\beq\label{def-normw}
\bega 
B(\gamma)=&\quad \sup_{0<\gamma t< \rho_0}\sup_{\rho<\rho_0- \gamma t}\Bigl{\{} 
 ||| \omega(t)|||_{\rho,\delta(t),1} +  ||| \omega(t)|||_{\rho,\delta(t),2}(\rho_0-\rho-\gamma t)^{\zeta}\Bigr{\}}\\
\enda 
\eeq 
with the boundary layer norm $$ ||| \omega(t)|||_{\rho,\delta(t),k}  = \sum_{j+\ell \le k} \|\partial_x^j (\psi(z)\partial_z)^\ell \omega(t)\|_{\rho,\rho,\delta(t)} .$$ Thanks to Lemma \ref{lem-bilinear}, we estimate 
\begin{equation}\label{non-est}
\begin{aligned}
 ||| N(t)|||_{\rho,\delta(t),0} &\lesssim  ||| \omega(t)|||_{\rho,\delta(t),1}^2 \le B(\gamma)^2
 \\
 ||| N(t)|||_{\rho,\delta(t),1} &\lesssim  ||| \omega(t)|||_{\rho,\delta(t),1}  ||| \omega(t)|||_{\rho,\delta(t),2}  \le B(\gamma)^2 (\rho_0-\rho-\gamma t)^{-\zeta}
. \end{aligned}\end{equation}
Now, using the Duhamel integral formula \eqref{Duh-w1}, we estimate 
$$
\begin{aligned}
||| \omega(t) |||_{\rho,\delta(t),k} &\le ||| e^{\nu tB}\omega_{0}|||_{\rho,\delta(t),k} +\int_{0}^{t} ||| e^{\nu(t-s)B} N (s)|||_{\rho,\delta(t),k} \; ds 
\\&\quad + \int_0^t ||| \Gamma(\nu (t-s))(\partial_z \Delta^{-1} N(s))_{\vert_{z=0}}|||_{\rho,\delta(t),k} \;ds.
\end{aligned}
$$
In view of Proposition \ref{prop-Stokes}, the term from the initial data is already estimated, giving $||| e^{\nu tB}\omega_{0}|||_{\rho,\delta(t),k} \le \|\omega_0\|_{\rho,\delta(0),k}$. In addition, using the estimates on $\Gamma(\nu t)$ from Proposition \ref{prop-Stokes} and \eqref{non-est}, we have 
$$
\begin{aligned}
&\int_0^t ||| \Gamma(\nu (t-s))(\partial_z \Delta^{-1} N(s))_{\vert_{z=0}}|||_{\rho,\delta(t),k}ds 
\\&\lesssim \int_0^t \sqrt t(t-s)^{-1/2} \| (\partial_z \Delta^{-1} N(s))_{\vert_{z=0}}\|_{\rho,k} \; ds + \sqrt \nu \int_0^t \| (\partial_z \Delta^{-1} N(s))_{\vert_{z=0}}\|_{\rho,k+1} \; ds
 \\&\lesssim \int_0^t \sqrt t(t-s)^{-1/2}  ||| N(s)|||_{\mathcal{W}^{k,1}_{\rho,\rho}} \; ds + \sqrt \nu \int_0^t ||| N(s)|||_{\mathcal{W}^{k+1,1}_{\rho,\rho}} \; ds
  \\&\lesssim t \sup_{0\le s\le t} ||| N(s)|||_{\mathcal{W}^{k+1,1}_{\rho,\rho}} 
  \end{aligned}$$
in which $\|N(s)\|_{\mathcal{W}^{k+1,1}_{\rho,\rho}}$ is bounded, thanks to the (independent) iteration obtained in the previous subsection. It remains to estimate 
$$\begin{aligned}
&\int_{0}^{t} ||| e^{\nu(t-s)B} N(s)|||_{\rho,\delta(t),1} \; ds 
\\&\le C_0 \int_0^t \sqrt{\frac{t}{s}}||| N(s)|||_{\rho,\delta(s),1} \; ds
\le C_0 B(\gamma)^2 \int_0^t \sqrt{\frac{t}{s}}  (\rho_0-\rho-\gamma s)^{-\zeta}\; ds
\\
&\le C_0 B(\gamma)^2 \Big( \int_0^{t/2}+ \int_{t/2}^t \Big) \sqrt{\frac{t}{s}}  (\rho_0-\rho-\gamma s)^{-\zeta}\; ds
\\
&\le C_0 B(\gamma)^2 \Big( t (\rho_0 - \rho - \frac12 \gamma t)^{-\zeta} + \frac1\gamma (\rho_0-\rho-\frac12\gamma t)^{1-\zeta}\Big)
\\
&\le C_0 \gamma^{-1} B(\gamma)^2 (\rho_0 - \rho)^{-\zeta} ,
\end{aligned}$$
in which we used $\gamma t\le \rho_0$ and $\gamma t < \rho_0 - \rho$. Next, noting that  $\rho < \rho_0 - \gamma t \le \rho_0 - \gamma s$, we take $\rho' = \frac{\rho + \rho_0 - \gamma s}{2}$ and compute 
$$\begin{aligned}
&\int_{0}^{t} ||| e^{\nu(t-s)B} N(s)|||_{\rho,\delta(t),2} \; ds 
\\&\le 
C_0 \int_0^t \sqrt{\frac{t}{s}}||| N(s)|||_{\rho,\delta(s),2} \; ds \le C_0 \int_0^t \sqrt{\frac{t}{s}} \frac{1}{\rho' - \rho}||| N(s)|||_{\rho',\delta(s),1}\; ds
\\&\le C_0 B(\gamma)^2 \int_0^t \sqrt{\frac{t}{s}}  (\rho_0-\rho-\gamma s)^{-1-\zeta}\; ds
\\
&\le C_0 B(\gamma)^2 \Big( \int_0^{t/2}+ \int_{t/2}^t \Big) \sqrt{\frac{t}{s}}  (\rho_0-\rho-\gamma s)^{-1-\zeta}\; ds
\\
&\le C_0 B(\gamma)^2 \Big( t (\rho_0 - \rho - \frac12 \gamma t)^{-1-\zeta} + \frac1\gamma (\rho_0-\rho-\gamma t)^{-\zeta}\Big)
\\
&\le C_0 \gamma^{-1} B(\gamma)^2 (\rho_0-\rho-\gamma t)^{-\zeta}.
\end{aligned}$$
This proves the boundedness of the iterative norm $B(\gamma)$, and hence the propagation of the boundary layer behaviors. Theorem \ref{theo-main} follows. 

\subsection{The inviscid limit}
Let us now prove Theorem \ref{theo-limit}. Since $\omega_0^\nu \in {\cal B}^{\rho_0,\sigma_0,\delta}$, the velocity $u_0^\nu$ and its conormal derivatives $\partial_x^k u_0^\nu, (\psi(z)\partial_z)^j u_0^\nu$, for $k,j,\ge 0$ and $\psi = \frac{z}{1+z}$, are all uniformly bounded on $\TT \times \RR_+$, thanks to the elliptic estimates; see Proposition \ref{inverseLaplace}. In particular, the limit of $u_0^\nu$ exists in the classical sense: $u^0(0) = \lim_{\nu\to 0}u_0^\nu$. From Theorem \ref{theo-main}, we check the validity of  the Kato's condition:
$$
\begin{aligned} 
&\nu \int_0^T\iint_{\TT \times \RR_+} |\nabla u(t,x,z)|^2 \; dxdz dt 
\\& = \nu \int_0^T\iint_{\TT \times \RR_+} |\omega(t,x,z)|^2 \; dxdz dt 
\\&\le C \nu \int_0^T\iint_{\TT \times \RR_+} e^{-2\beta z} \Bigl( 1 + \delta_t^{-1} \phi_{P} (\delta_t^{-1} z)  +  \delta^{-1} \phi_{P} (\delta^{-1} z)  \Bigr)^2\; dxdz dt 
\\&\le C \nu \int_0^T \Bigl( 1 + \delta_t^{-1} +  \delta^{-1} \Bigr) dt \le C_T \sqrt \nu
\end{aligned}$$ 
which tends to zero as $\nu \to 0$. This proves that $\sup_{t\in [0,T]}\| u^\nu(t) - u^0(t)\|_{L^2} \to 0$. The $L^p$ convergence follows from the interpolation between $L^2$ and $L^\infty$ norms and the fact that $u^\nu$ is bounded in $L^\infty$ thanks to the elliptic estimate \eqref{laplace-5}. 

\section{Compliance with Ethical Standards}

{\bf Conflict of Interest:} {\em The authors declare that they have no conflict of interest. This research does not involve Human Participants and/or Animals. The manuscript complies to the Ethical Rules applicable for the Archive for Rational Mechanics and Analysis journal. }

~\\
{\bf Acknowledgement.} {\em The authors are indebted to Claude Bardos and Emmanuel Grenier for their many fruitful discussions. 
The research was supported in part by the NSF under grants DMS-1405728} \& {\em DMS 1764119.}

\bibliographystyle{abbrv}

\def\cprime{$'$} \def\cprime{$'$}

\end{document}